\documentclass[leqno,a4paper]{article}
\usepackage{amsmath,
amsthm,amssymb}
\usepackage{eufrak}

\newcommand{\mT}{\mathcal{T}}

\newcommand{\R}{{\mathbb R}}
\newcommand{\oo}{{\underline{\omega}}}
\newcommand{\on}{{\underline{n}}}
\newcommand{\oee}{{\underline{e}}}
\newcommand{\ovv}{{\underline{v}}}

\newcommand{\ut}{{\underline{t}}}
\newcommand{\utau}{{\underline{\tau}}}

\newcommand{\pa}{{\partial}}

\newcommand{\Om}{{\Omega}}
\newcommand{\Ga}{{\Gamma}}
\newcommand{\vp}{{\varphi}}

\newcommand{\te}{{\theta}}

\newcommand{\ou}{{\underline{u}}}

\newcommand{\n}{{\underline{n}}}

\newcommand{\be}{\begin{equation}}
\newcommand{\ee}{\end{equation}}
\newcommand{\ba}{\begin{array}}
\newcommand{\ea}{\end{array}}
\def\vs1{\vspace{1ex}}

\newtheorem{theorem}{Theorem}[section]

\newtheorem{remark}{Remark}[section]

\newtheorem{problem}{Problem}[section]
\numberwithin{equation}{section}
\begin{document}
\title{\bf\normalsize THE 3-D INVISCID LIMIT RESULT UNDER SLIP BOUNDARY CONDITIONS. A NEGATIVE
ANSWER}
\author{by H.~Beir\~ao da Veiga and F.~Crispo}
\date{}
\maketitle \textit{\phantom{aaaaaaaaaaaaaaaaaaaaaaaaaa}}%

\begin{abstract}
We show that, \emph{in general}, the solutions to the
initial-boundary value problem for the Navier-Stokes equations under
a widely adopted Navier-type slip boundary condition do not
converge, as the viscosity goes to zero (in any arbitrarily small
neighborhood of the initial time), to the solution of the Euler
equations under the classical zero-flux boundary condition, and same
smooth initial data. Convergence does not hold with respect to any
space-topology which is sufficiently strong as to imply that the
solution to the Euler equations inherits the complete slip type
boundary condition (see the Theorem \ref{theocorol} below). In our
counter-example $\,\Om\,$ is a sphere, and the initial data may be
infinitely differentiable. The crucial point here is that the
boundary is not flat. In fact (see \cite{bvcrgr}), if $\,\Om
=\,\R^3_+\,,$ convergence holds in $C([0,T]; W^{k,p}(\R^3_+))$, for
arbitrarily large $k$ and $p$. For this reason, the negative answer
given here was not expected.
\end{abstract}

\bibliographystyle{amsplain}

\section{Introduction}

In some recent papers, see \cite{bvcr}, \cite{bvcr2}, \cite{bvcrgr},
we have considered the problem of the strong convergence up to the
boundary, as $\nu \rightarrow \,0\,,$ of the solutions $\,u^\nu\,$
of the Navier-Stokes equations in the cylinder $\Omega\times (0,T)$
\begin{equation}
\left\{
\begin{array}{l}
\partial_t\,\ou^\nu +\,(\ou^\nu \cdot\,\nabla)\,\ou^\nu -\,\nu\,\Delta\,\ou^\nu +\,\nabla\,\pi=\,0,\\
{div}\,\ou^\nu =\,0\,,\\
\ou^\nu(0)=\,\ou_0\,,
\end{array}
\right. \label{1;4}
\end{equation}
under the slip boundary conditions at $\partial\Omega\times (0,T)$
\begin{equation}
\left\{
\begin{array}{l}
\ou^\nu\cdot \on=\,0,\\
\oo^\nu \times \,\on=\,0\,,
\end{array}
\right. \label{bcns}
\end{equation}
where $\,\oo=\, curl \,\ou\,$, to the solution $u$ of the Euler
equations
\begin{equation}
\left\{
\begin{array}{l}
\partial_t\,\ou +\,(\ou\cdot\,\nabla)\,\ou+\,\nabla\,\pi=\,0,\\
{div}\,\ou=\,0\,,\\
\ou(0)=\,\ou_0\,,
\end{array}
\right. \label{1;4e}
\end{equation}
under the zero flux boundary condition
\begin{equation}
\ou\cdot\,\n=\,0\,.%
\label{bceu}
\end{equation}
The domain $\Om\,$ is an open set in $\R^3\,$ locally situated on
one side of its boundary $\Ga$, and $\on=\,(n_1,\,n_2,\,n_3)\,$ is
the unit outward normal to $\Ga$. We have showed that strong
convergence holds provided that the boundary is flat. In particular,
in the half-space case  we proved \cite{bvcrgr} that if
the initial data are in $W^{k,p}(\R^3_+)$, then convergence holds
in $C([0,T]; W^{k,p}(\R^3_+))$, for arbitrarily large $k$ and $p$.
Moreover, a minimal set of independent, necessary and sufficient,
compatibility conditions on $\,\Ga\,$ at $\,t=\,0\,$ is displayed.
These conditions appear only if $k\geq 4$.%

\vspace{0.2cm}

The natural next step is  to study if and how the above results
continue to hold in the presence of non-flat boundaries. As a matter
of fact, in the two-dimensional case the answer turns out to be
positive; see, for instance, \cite{clopeau}. In the three
dimensional case, the strong inviscid limit appears, instead, to be
a much more complicated issue and, so far, an open problem; see
\cite{bvcr} for a quite complete discussion on this
problem, and for proofs of related useful equations.\par%
In the recent paper \cite{xin} an interesting new approach to the
problem is introduced. Notwithstanding, the method of proof only
fully works if the boundary is flat. This fact was pointed out in
the subsequent papers \cite{bvcr} and \cite{bvcr2} where it was
emphasized that the non-flat boundary problem remains still
unsettled; for a review on these results see also \cite{CRParma}.

\vspace{0.2cm}

Objective of this note is to show that a strong inviscid limit result, in the
presence of non-flat boundaries, is false in general. Roughly
speaking by ``strong'' we mean that it is taken in function spaces such that all
the derivatives that appear in the equations, including the boundary
conditions, are integrable. In particular the result is false in
general,  when $\Omega$ is the unit sphere, and for
$C^{\infty}(\overline{\Om})\,$ divergence free initial data which
satisfies the slip boundary conditions \eqref{bcns}. For instance,
as $\nu$ tends to zero, the solutions to the Navier-Stokes equations
do not converge to the solution of the Euler equations in $L^1(0,
t_0; W^{s,q})$, for any arbitrarily small $\,t_0>\,0\,$, any $q\geq
1\,,$ and any $s>\,1+\,\frac1q $. Note that the above unique
solution to the Euler equations is infinitely differentiable, and
the above solutions to the Navier-Stokes equations are ``smooth''.%

\vspace{0.2cm}

A $W^{2,\,p}$ vanishing viscosity limit result in general domains
was recently claimed  in \cite{spirito}, Theorem 1.1. In the first
part of this preprint, the authors review methods and arguments
previously introduced and developed in references \cite{bvcr} and
\cite{bvcr2}. After this review, the authors go to the proof of the
main result, their Theorem 1.1. In doing this, they partially appeal
to some general ideas developed in a sequence of papers by one of
us, introduced to study sharp singular limit problems. In fact, this
approach in the present context seems to us a good choice. Actually,
the layout of the paper is convincing. Unfortunately, the final
result is incompatible with the counter-example presented below.

\vspace{0.2cm}

\begin{remark}
\rm{On flat portions of the boundary, the slip boundary conditions
coincide with the classical Navier boundary conditions
\begin{equation}
\left\{
\begin{array}{l}
\ou\cdot \n=\,0,\\
\ut \cdot \,\utau=\,0\,,
\end{array}
\right. \label{boundary-ho}
\end{equation}
where $\,\utau\,$ stands for any arbitrary unit tangential vector.
Here $\,\ut\,$ is the stress vector defined by $\,\ut
=\,\mT\cdot\,\n\,,$ where the stress tensor $\,\mT\,$ is defined by
\begin{equation*}
\mT =-\pi\,I+\,\frac{\nu}{2} \,(\nabla \ou+\nabla \ou^T)\,.
\end{equation*}
These conditions were introduced by Navier in 1823 and derived by
Maxwell in 1879 from the kinetic theory of gases. In the general
case
\begin{equation}
\ut \cdot \,\utau=\,\frac{\nu}{2} \,(\oo \times \,\n) \cdot \,\utau
\,-\,\nu \,\mathcal{K}_{\tau}\, \ou \cdot \,\utau\,,%
\label{parec}
\end{equation}
where $\,\mathcal{K}_{\tau}\,$ is the principal curvature in the
$\,\utau\,$ direction, positive if the corresponding center of
curvature lies inside $\,\Om\,$.\par%
Note that our counter-example does not exclude that strong
vanishing results hold under the Navier boundary conditions
in the non-flat boundary case.}%
\label{tolosa}
\end{remark}

\vspace{0.2cm}

We end the introduction by stating the following two theorems.
\begin{theorem}
Let $\Omega=\,\{x:\,|\,x\,|<\,1\,\}\,,$ be the $3$-dimensional
unitary sphere. There is an explicit family (see the Theorem
\ref{contra}) of $C^{\infty}(\overline{\Om})\,$, divergence free
initial data $\,\underline u_0\,$, which satisfies the slip boundary
conditions \eqref{bcns}, and such that the following holds. Given an
element $\,\underline u_0\,$ belonging to the above family, there
exists a $\,t_0
>\,0\,$ such that the corresponding (unique, indefinitely differentiable) local solution
$\,\underline u(t)\,$ to the Euler equations \eqref{1;4e},
\eqref{bceu} does not
satisfy the boundary condition $\oo \times \,\on=\,0 $, for any $\,t \in (0,\,t_0]\,.$%
\label{theointro}
\end{theorem}
In particular, the following result holds.
\begin{theorem}
Let $\,\underline u_0\,$ be a given, fixed, initial data belonging
to the class referred in the above theorem \ref{theointro}. Denote
by $\,\ou^\nu\,$ the $\nu-$family of solutions to the Navier-Stokes
equations \eqref{1;4}, \eqref{bcns} with initial data $\,\underline
u_0\,$, and denote by $\,\underline u\,$ the solution of the Euler
equations
\eqref{1;4e}, \eqref{bceu} with initial data $\, \underline u_0\,$.\par%
There does not exist a $\,t_0 >\,0\,$ and exponents $\,q \geq \,1\,$
and $\,s>\,1+\,\frac{1}{q}\,$ such that $\,\ou^\nu\,$ converges to
$\,\underline u\,$ in $\,L^1(0,\,t_0; W^{s, \,q}(\Om)\,)\,.$ The
particular
case $L^1(0, t_0;\,W^{2,1}(\Om)\,)\,$ is also included in this statement.%
\label{theocorol}
\end{theorem}
\begin{remark}
{\rm Actually the convergence in the above theorem \ref{theocorol}
fails for any arbitrary subsequence, even under weaker convergence
hypotheses.}%
\label{remgen}
\end{remark}
\emph{Plan of the paper}: In section \ref{reduction} we show how to
turn the proofs of the above two theorems into the construction of a
suitable class of vector fields (called here ``counter-examples").
In section \ref{contraex} we explicitly construct the above vector
fields.
\section{Reduction to a functional problem in space variables}\label{reduction}%
In spite of the exceptionally strong convergence results in the case
of flat boundaries, at a certain point we became inclined to believe that a
strong inviscid limit result is false in general. This guess led us
to look for a counter-example, by reductio ad absurdum, as follows.
Let  $\underline u_0$ be a smooth divergence free initial data,
which satisfies the slip boundary conditions \eqref{bcns}, and
denote by $\,\ou^\nu\,$ and $\,\underline u\,$ the corresponding
solutions to the above Navier-Stokes and Euler boundary value
problems. Moreover, assume (per absurdum) that $\,\ou^\nu\,$
converges to $\, \underline u\,$ as $\,\nu $ goes to zero, with
respect to a specific $\tau-$topology, which (by assumption) is
sufficiently strong as to imply that the limit $\,\underline u(t)\,$
inherits the boundary condition $\,\oo^\nu \times \,\on=\,0\,$ near
$t=\,0\,$ (for instance, convergence in $L^1(0, t_0; W^{2,1})\,$).
This would imply that the Euler equations \eqref{1;4e} under the
classical boundary condition \eqref{bceu} necessarily enjoy the
following \emph{persistency property}: if a smooth initial data
satisfies the additional boundary condition $ \oo(0) \times
\,\on=\,0\,,$ then at least for small times, $\,\oo(t)\,$ must
verify this same property (we note that this was also considered as
an open problem). It follows that, in order to contradict the
possibility of the above $\tau-$convergence result, it is sufficient
to contradict the above persistency property for the Euler
equations. Next, by arguing as follows, we turn the proof of the
absence of the above persistency property into a problem concerning
only the space variables. External multiplication of the Euler
vorticity equation by the normal $\,\on\,$, point-wise on $\Gamma$,
leads to the equation
\begin{equation} \pa_t\,(\oo \times \, \on \,)-\,curl\,(\ou \times
\oo)\,\times \on =\,0\,.
\end{equation}
If the persistency property  holds, the first term in the above
equation must vanish identically on $\,\Gamma\,,$ at time
$\,t=\,0\,$. Hence the second term must verify the same property,
say
\be\label{cc}%
curl \,(\ou_0\times \oo_0)\times \on=0%
\ee%
on  $\,\Gamma\,$.
\par
Consequently, in order to prove that the above persistence property
does not hold and, a fortiori, that the above $\tau-$inviscid limit
result does not hold in general, it is sufficient to solve the
following problem.
\begin{problem}
To exhibit a smooth, divergence free vector field $\,\underline
u_0\,$, in a bounded, regular, open set $\Om$, which satisfies the
slip boundary conditions everywhere on $\,\Ga\,$, but does not
satisfy, somewhere on $\,\Ga\,,$ the boundary condition \eqref{cc}.
\end{problem}
Below, we succeed in constructing, \emph{globally} in $\,\Om \,,$ a
wide class of $C^{\infty}(\overline{\Om})\,$ vector fields for which
the above, negative, result holds. We assume $\Om$ to be the
$3$-dimensional unitary sphere and  display our vector field in
spherical coordinates. Once the vector fields are known, the
verification of the desired properties is straightforward.

\section{The counter-example}\label{contraex}%
In what follows we use spherical coordinates $(r,\te, \vp)$. For any
vector field $\ou$, we denote by $\,u_r$, $\,u_\te$ and $\,u_\vp$
the components of $\ou$ in the orthonormal, positively oriented,
local basis $\left(\oee_r,\,\oee_\te, \,\oee_\vp\right)$ . Just for
convenience, let us recall the expressions of $\nabla \,\cdot\, \ou$
and $\,\oo$ in this curvilinear coordinate system: \be\label{divu}
\begin{array}{ll}\displaystyle \vs1 \nabla \cdot
\ou=\,\frac{1}{r^2}\,\frac{\pa}{\pa r}\,(r^2\, u_r)+\,
\frac{1}{r\,\sin \te}\,\frac{\pa}{\pa \te}( u_\te\,\sin \te)+\,
\frac{1}{r\,\sin \te} \frac{\pa u_\vp}{\pa \vp};
\end{array}\ee
\be\label{rotu}
\begin{array}{ll}%
\displaystyle \vs1 curl\,\ou=\,\frac{1}{r\sin
\te}\,\left(\frac{\pa}{\pa \te} (u_\vp \,\sin \te)-\, \frac{\pa
u_\te}{\pa \vp}\,\right) \oee_r\\\hfill \displaystyle +
\frac{1}{r}\,   \left(\frac{1}{\sin \te}\,\frac{\pa u_r}{\pa \vp}
-\, \frac{\pa}{\pa r} (r\, u_\vp) \right) \oee_\te+ \frac{1}{r}\,
\left(\frac{\pa}{\pa r} (r\, u_\te)-\, \frac{\pa u_r}{\pa \te}
\right) \oee_\vp\,.%
\end{array}\ee
We also recall that, for a scalar field $f=f(r,\te,\vp)\,$,
\begin{equation}
\nabla \,f=\, \frac{\pa \,f}{\pa r}\,\oee_r+\,\frac{1}{r}\,\frac{\pa
\,f}{\pa \te}\,\oee_\te +\,\frac{1}{r\,\sin \te}\,\frac{\pa \,f}{\pa
\vp} \oee_\vp\,.%
\end{equation}
We consider the $3$-dimensional unitary sphere
$\Omega=\,\{x:\,r<\,1\,\}\,,$ and denote by $\Gamma$ its boundary.
The unit external normal  is denoted by $\on$. Clearly $\on=\oee_r$
on $\Gamma$.
\par
Let $h(r)$ be a $C^{\infty}\left([0,+\infty)\right)$ real function,
and $g(\te,\vp)$ be a $C^{\infty}([0,\pi]\times \R)$ real function,
$2\pi$-periodic on $\vp$. Just for convenience, we assume that
$h(r)$ vanishes in a neighborhood of $r=0$ and $g(\te,\vp)$ vanishes
for $\te$ in a neighborhood of $\te=0$ and $\te=\pi$ (and arbitrary
$\vp$). Set
$$G(\te,\vp)= \frac{\pa}{\pa \te}\left( \sin\te\frac{\pa
g}{\pa\te}\right)+\frac{1}{\sin \te}\frac{\pa^2 g}{\pa\vp^2}.$$

\begin{theorem}
Let $\ou$ be the vector field \be\label{exu}
\ou=\,-\,\frac{h(r)}{\sin \te}\,\frac{\pa g}{\pa \vp}\, \underline
e_{\te}+ h(r)\,\frac{\partial g}{\pa \te}\, \oee_{\vp}\,.\ee Then
the following results hold:
\begin{itemize}
\item[i)] $\nabla \cdot \ou=0\,$ in $\,\Om\,,$ $\ou\cdot \on=0\,$ on
$\Gamma$.
\item[ii)] If $\,h(1)+h'(1)=0\,$, then $\,\oo\times \on=0\,$ on $\,\Gamma$.
\item[iii)] If $\,h(1)+h'(1)=0\,$, with $h(1)\not=0$, and if
\be\label{a1}\frac{\pa g}{\pa \vp}\,\not=\,0\quad \mbox{ and }\quad
G(\te,\vp)\not=0 \ee at a point $P$ on $\Gamma$, then $\left[\,curl
(\ou\times \oo)\right]_\te\,\not=\,0$ in a neighborhood of $P$.
Similarly if $\,h(1)+h'(1)=0\,$, with $h(1)\not=0$ and if
\be\label{a2} \frac{\pa g}{\pa \te}\,\not=\,0\quad \mbox{ and }\quad
G(\te,\vp)\not=0 \ee at a point $P$ on $\Gamma$, then $\left[\,curl
(\ou\times \oo)\,\right]_\vp\,\not=\,0$ in a neighborhood of $P$.
\end{itemize}
\label{contra}
\end{theorem}

\begin{proof}
Claims in $i)$ follow by a straightforward calculation, using
\eqref{divu} and recalling that $\on=\,\oee_r$  on $\Gamma$.
\vskip0.2cm By using \eqref{rotu}, and by observing that \eqref{exu}
yields $\displaystyle u_r=\frac{\pa u_r}{\pa \te}=\frac{\pa u_r}{\pa
\vp}=0$ in $\Omega$, we show that $\oo$ is given in
$\overline{\Omega}\,$ by
$$\begin{array}{ll}
\displaystyle\vs1
\oo=\omega_r\,\oee_r+\omega_\te\,\oee_\te+\omega_\vp\,\oee_\vp\\
\hfill=\displaystyle\frac{h(r)}{r\sin \te}\,G(\te,\vp)\,
\oee_r-\,\frac1r\,\frac{\pa}{\pa r}\,(r\,h(r))\,\frac{\pa g}{\pa
\te}\, \oee_\te- \frac{1}{r\sin \te}\,\frac{\pa}{\pa
r}\,(r\,h(r))\,\frac{\pa g}{\pa \vp}\, \oee_\vp\,.\end{array}$$ In
particular, on $\Gamma$ the vector field $\oo\times \on$ is given by
$$\oo\times \on=\,\omega_\vp\,\oee_\te-\omega_\te\,\oee_\vp=\,-\,\frac{1}{r\sin \te}\,\frac{\pa}{\pa
r}\,(r\,h(r))\,\frac{\pa g}{\pa \vp}\,\oee_\te +\,
\frac1r\,\frac{\pa}{\pa r}\,(r\,h(r))\,\frac{\pa g}{\pa
\te}\,\,\oee_\vp\,.$$ Therefore, if $\frac{\pa}{\pa
r}\,(r\,h(r))|_{r=1}=0$, we get $\oo\times \on=0$ on $\Gamma$. This
proves ii).\vskip0.2cm Let us pass to the last point $iii)$. From
the previous steps, we have
\be\label{uob}u_r=\omega_\te=\omega_\vp=0\ \mbox{ on }\ \Gamma\,.\ee

Set $\ovv=\ou\times \oo$. Since $\,u_r=\,0\,$ in $\Omega$, $\ovv$ is
given by \be\label{vg}
\ovv=(u_\te\,\omega_\vp\,-\,u_\vp\,\omega_\te)\, \oee_r+
\,u_\vp\,\omega_r\, \oee_\te -\,u_\te\,\omega_r\, \oee_\vp.\ee Note
that $\oo\times \on=0$ on $\Gamma$ implies that $\ovv$ is tangential
to $\Gamma$. Hence, \be\label{vong} v_r= \frac{\pa v_r}{\pa
\te}=\frac{\pa v_r}{\pa \vp}=0\quad \mbox{ on }\, \Gamma.\ee
Further, from \eqref{uob}, it follows
$$v_\te=u_\vp\,\omega_r\ \mbox{ and \ }  v_\vp=-u_\te\,\omega_r,
\quad\mbox{ on }\Gamma\,.
$$  By recalling \eqref{rotu} and then
using \eqref{uob}, \eqref{vg} and \eqref{vong}, we show that the
$\te$ and the $\vp$ components of $curl \,\ovv$ on $\Gamma$ are
given by
$$
\left[\,curl \,\ovv\,\right]_\te= -\,\frac{1}{r}\,\frac{\pa }{\pa
r}\,(r\,v_\vp)
$$
and
$$\left[\,curl \,\ovv\,\right]_\vp= \,\frac{1}{r}\,\frac{\pa }{\pa
r}\,(r\,v_\te)\,,
$$
respectively.\par%
Straightforward calculations lead to
$$\left[\,curl
\,\ovv\,\right]_\te= -\frac{2}{\sin^2\te} \,h(1)h'(1)\frac{\pa
g}{\pa \vp}\, G (\te,\vp)\ \mbox{ on }\Gamma.$$ Therefore, if
$h(1)\not=0$ (hence $h'(1)\not=0$ by $\,h(1)+h'(1)=0\,$) and if
\eqref{a1} is satisfied at some point $P\in \Gamma$, it follows that
$\left[\,curl \,\ovv\,\right]_\te\not=0$ at $P$. Consequently this
last quantity does not vanish in a neighborhood of $P$. The same
arguments applied on the $\vp$-component of $curl \,\ovv$ on
$\Gamma$ ensure that under condition \eqref{a2} at some point $P$,
$\left[\,curl \,\ovv\,\right]_\vp \not=0$ at $P$.
\end{proof}

{\it Acknowledgments.}\; The work of the second author was supported
by INdAM (Istituto Nazionale di Alta Matematica) through a Post-Doc
Research Fellowship at Dipartimento di Matematica Applicata,
University of Pisa.

\end{document}